\documentclass[a4paper,12pt,reqno]{amsart}

\usepackage[utf8]{inputenc}
\usepackage{amsmath, amsthm, amsfonts, amssymb}
\usepackage[foot]{amsaddr}
\usepackage{stmaryrd}	
\usepackage{graphicx}
\usepackage{xcolor}
\usepackage{tikz}
\usepackage{paralist}
\usepackage[hidelinks]{hyperref} 
\usepackage{pb-diagram}
\usepackage{float}

\usepackage{todonotes}

\newcommand{\N}{\mathbb{N}}

\newcommand{\R}{\mathbb{R}}
\newcommand{\C}{\mathbb{C}}

\newcommand{\abs}[1]{\left|#1\right|}

\newcommand{\be}{\begin{equation}}
\newcommand{\ee}{\end{equation}}

\newcommand{\lin}{\ensuremath{\mathrm{lin}}}

\newtheorem{theorem}{Theorem}
\newtheorem{lemma}[theorem]{Lemma}

\newtheorem{prop}[theorem]{Proposition}

\theoremstyle{definition}
\newtheorem{remark}[theorem]{Remark}

\newcommand{\convH}{\operatorname{conv}(\mathcal{H})}

\title[Sampling projections]{Sampling projections in the uniform norm 
}

\author{David Krieg$^1$, Kateryna Pozharska$^{2,3}$, \\ Mario Ullrich$^4$, Tino Ullrich$^3$
}

\date{\today}

\address{
$^1$Faculty of Computer Science and Mathematics, University of Passau, Germany
{$^2$Institute of Mathematics of NAS of Ukraine, Kyiv, Ukraine}
{$^3$Faculty of Mathematics, Chemnitz University of Technology, 
Germany}
{$^4$Institute of Analysis \& Department of Quantum Information and Computation, 
Johannes Kepler University Linz, Austria}}
\email{
david.krieg@uni-passau.de} 
\email{pozharska.k@gmail.com}
\email{mario.ullrich@jku.at}
\email{tino.ullrich@mathematik.tu-chemnitz.de}

\allowdisplaybreaks


\begin{document}

\begin{abstract}
We show that there are sampling projections onto arbitrary 
$n$-dimensional subspaces of $B(D)$
with at most $2n$ samples and norm of order~$\sqrt{n}$, where $B(D)$ is the space of bounded functions on a set $D$. 
This gives a more explicit form of the Kadets-Snobar theorem for the uniform norm.
We discuss consequences for optimal recovery in~$L_p$. 
\end{abstract}

\maketitle

\section{Introduction}


The \emph{theorem of Kadets and Snobar}~\cite{KS}, 
see also~\cite[6.1.1.7]{Pie07}, 
asserts that for any $n$-dimensional subspace $V_n$ 
of a normed space $G$, there is a 
linear operator
$P\colon G\to V_n$ 
with $Pf=f$ for $f\in V_n$, i.e., $P$ is a \emph{projection onto} $V_n$, and
$\|P\|\le\sqrt{n}$. 
As a consequence 
$P$ is a linear mapping
that 
returns an approximation $Pf \in V_n$ to any element $f\in G$ 
whose error exceeds
the error of best approximation 
 by a factor of at most $1+\sqrt{n}$\,: 
For any $g\in V_n$, 
\[
 \Vert f - Pf \Vert_G \,\le\, 
 \Vert f - g \Vert_G
 + \Vert P(f - g) \Vert_G
 \,\le\, \big(1+\|P\|\big)\, \Vert f - g \Vert_G.
\]
A
best approximation,
on the other hand,
is in general not expressed by a linear mapping
of $f$.

Unfortunately, it is not clear what \emph{information} of 
$f\in G$ is required to compute its projection {$Pf \in V_n$}.
Here, we consider the case where $f$ is a function and only a finite number of function evaluations of $f$ are allowed as information. We therefore restrict to 
$G=B(D)$, i.e., the space of all bounded complex-valued functions on a set~$D$ 
equipped with the sup-norm $\|f\|_{\infty}:= \sup_{x\in D}|f(x)|$, 
and seek for \emph{sampling projections} of the form
\[
 P_m \colon B(D)\to V_n,
 \quad P_mf=\sum_{i=1}^m f(x_i)\, \varphi_i,
\]
for some natural number $m\ge n$, $x_i\in D$ and $\varphi_i\in B(D)$.
Ideally, the number $m$ of samples should be close to the dimension $n$ of the space $V_n$.
In this setting, we are only aware of 
a lemma by 
Novak~\cite[{Lemma~1.2.2}]{Novak}, 
a version of \emph{Auerbach's lemma} that implies that there is 
a sampling projection $P_m$ with $m=n$ and
$\|P_m\|\le 
n +1
$, 
see Remark~\ref{rem:interpol}.

The main goal in this paper is to show that also 
the projection in the Kadets-Snobar theorem can be given by a 
sampling projection and thus improving on Novak's result by a factor of order $\sqrt{n}$. 
However, in contrast to 
Novak, we 
require a slight 
\emph{oversampling}, i.e., we need $m\ge cn$, $c>1$, 
function evaluations for subspaces $V_n$ of dimension $n$. 
%

\begin{theorem}\label{thm:KS-new}
There is an absolute constant $C>0$ such that the following holds. Let $D$ be a set and $V_n$ be an $n$-dimensional subspace of $B(D)$. 
Then there are $2n$ points $x_1,\dots,x_{2n}\in D$ 
and functions $\varphi_1,\dots,\varphi_{2n} \in V_n$ such that 
$P\colon B(D)\to V_n$  
with 
$Pf=\sum_{i=1}^{2n} f(x_i)\, \varphi_i$  
is a projection onto $V_n$ with $\|P\| \,\le\, C\sqrt{n}$.
\end{theorem}

\medskip

Theorem~\ref{thm:KS-new} is sharp in the following sense:
\begin{itemize}
    \item Using $m=\mathcal{O}(n)$ samples, 
    the norm bound $C\sqrt{n}$ 
    cannot be replaced with a lower-order term, see Section~{\ref{sec:sharpness}.}
    \item If we want to use $m=n$ samples,
    the norm bound $C\sqrt{n}$ has to be replaced by a linear term in $n$,
    see~\cite[Rem.~1.2.3]{Novak}. 
\end{itemize}
The oversampling factor $2$ can be replaced by any constant {$c>1+1/n$},
see Remark~\ref{rem:const}.
It is an interesting open problem, what the smallest possible norm is when using only $n+n^\alpha$ samples with $0\le \alpha <1$.

\medskip

Theorem~\ref{thm:KS-new} is based on a specific type of 
\emph{discretization} of the uniform norm, 
which might be of independent interest. 
In particular, Theorem~\ref{thm:KS-new} will be deduced from the following
{and Lemma~\ref{lem:sampling-projections-general}}, 
see Section~\ref{sec:LS}.

\begin{theorem}\label{thm:disc}
Let $D$ be a set and $V_n$ be an $n$-dimensional subspace of $B(D)$. 
Then there are 
points $x_1,\dots,x_{2n}\in D$ 
such that, for all $f\in V_n$, we have
\begin{equation}\label{eq:disc-intro} 
 \|f\|_{\infty} \;\leq\; 
 42
 \, 
 \Big(\sum_{k=1}^{2n} \abs{f(x_k)}^2 \Big)^{1/2}.
\end{equation}
\end{theorem}

\medskip

Theorem~\ref{thm:disc} bounds the continuous sup-norm by a discrete $\ell_2$-norm.
It might appear more natural to the reader to compare the sup-norm 
directly with a discrete sup-norm. 
This is a much more studied subject, see Remark~\ref{rem:uniform-disc}. 
Clearly, \eqref{eq:disc-intro} also gives
\begin{equation}\label{eq:disc-uniform}
 \|f\|_{\infty} \,\leq\, c \sqrt{n}\ \max_{k\le 2n}\, \abs{f(x_k)}
\end{equation}
with the new constant $c=\sqrt{2} C$, 
with $C=42$ from Theorem~\ref{thm:disc},
see also Remark~\ref{rem:const}
regarding constants.
However, working directly 
with the discrete $\ell_2$-norm
is crucial in our proof of Theorem~\ref{thm:KS-new} and
seems to be a key novelty in our approach.

\medskip

The first ingredient in the proof of Theorem~\ref{thm:disc} is a result by Kiefer and Wolfowitz from the classical paper \cite{KieWo60},
which we discuss and refine in Section~\ref{sec:KW}.
It provides a discrete probability measure $\varrho$ on $D$ 
with at most $n^2+1$ nodes 
such that $\Vert \cdot \Vert_\infty \le \sqrt{n} \Vert \cdot \Vert_{L_2(\varrho)}$ on the space $V_n$.
The second ingredient is a subsampling approach from~\cite{BSS},
in the form of a result 
from \cite{BSU}.

\medskip

With the above type of discretization,
we are also able to give an explicit expression 
for the sampling projection of Theorem~\ref{thm:KS-new}: 
an unweighted least squares method. 
To be precise, if $X \subset D$ is a finite multiset 
(i.e., a set where the same element is allowed to appear multiple times), 
we define a seminorm $\Vert \cdot \Vert_X$ on $B(D)$ via 
\begin{equation}\label{eq:def-seminorm}
  \Vert f \Vert_X^2 \,:=\, \frac{1}{|X|} \sum_{x\in X} |f(x)|^2.
\end{equation}
The (unweighted) least-squares algorithm $A(X,V)\colon B(D) \to V$
on a finite-dimensional subspace $V\subset {B(D)}$ is then defined by
\[
  A(X,V) f \,:=\,  \underset{g \in V}{\rm argmin}\ \Vert g-f \Vert_X.
\]
It is a well-defined linear operator if and only if $\Vert \cdot\Vert_X$ is a norm on $V$. 
In this case, it is clearly a projection onto $V$.

\medskip 

{The discretization result from Theorem~\ref{thm:disc} 
implies the following error bound for least squares approximation in the uniform norm, see Lemma~\ref{lem:general-norm}.}

\begin{theorem}\label{thm:LS_sup}
    There is a universal constant $C>0$ such that the following holds. Let $D$ be a 
    set 
    and $V_n$ be 
an $n$-dimensional subspace of $B(D)$.  
    Then there exists a multiset $X\subset D$ 
    of cardinality at most $2n$ such that, for all $f\in B(D)$,
    \[
     \Vert f - A(X,V_n) f \Vert_\infty
     \,\le\, C \sqrt{n} \cdot \inf_{g\in V_n} \Vert f - g \Vert_\infty.
    \]
\end{theorem}

\smallskip

Related
results have been obtained under the assumption that the used point sets are suitable for the discretization in the uniform norm, see e.g.~\cite[Thm.~2]{CL08}, and Remark~\ref{rem:uniform-disc}. 
Moreover, we can use Theorem~\ref{thm:disc} also in combination with \cite[Thm.~2.1]{T20} to get a version of Theorem~\ref{thm:LS_sup} with the least-squares estimator replaced by an estimator minimizing the discrete sup-norm. But this method is in general not linear.
\medskip

The remainder of the paper is structured as follows:
\begin{itemize}
    \item In Section~\ref{sec:KW}, we discuss the extremal probability measure 
    originating in \cite{KieWo60}.
    \item In Section~\ref{sec:disc}, we prove our 
    results on discretization.
    \item In Section~\ref{sec:LS}, we study the error of least-squares algorithms.
    \item In Section~\ref{sec:OR}, we discuss consequences for the problem of optimal sampling recovery and the sharpness of our results.
\end{itemize}
\medskip

All our results have extensions to the spaces 
$L_p(\mu):=L_p(D,\mu)$,
where $\mu$ is an arbitrary probability measure on $D$ and $1\le p \le \infty$. 
In particular, we will prove in Theorem~\ref{thm:widths} the relation
\[
 g_{4n}^{\rm lin}(F,L_p(\mu))
 \,\le\, C\,n^{(1/2-1/p)_+}\,d_n(F,B_\mu(D)),
\]
between the linear sampling numbers and the Kolmogorov numbers of a class $F\subset B_\mu(D)$,
the space of bounded measurable functions on $D$, see Section~\ref{sec:OR} for the definitions and details.

\goodbreak

\medskip

\begin{remark}[Auerbach's lemma and interpolating projections]\label{rem:interpol}
Auerbach's lemma is a classical result that asserts that, 
for general normed spaces $G$ and all $n$-dimensional subspaces $V_n\subset G$, 
there exists a basis $\varphi_1,\dots,\varphi_n\in V_n$ and linear functionals 
$\ell_1,\dots,\ell_n$ on $G$ with $\ell_i(\varphi_j)=\delta_{i,j}$ and $\|\ell_i\|=\|\varphi_i\|=1$ for $i,j=1,\dots,n$, 
see e.g.~\cite[5.6.2.4]{Pie07}, also for some history of the result.
This implies that 
$Pf:=\sum_{i=1}^n\ell_i(f) \varphi_i$ is a projection onto $V_n$ with norm at most $n$.
A sampling version of Auerbach's lemma has been given
{in \cite[1.2.2]{Novak},
where it is shown that}, for $G=B(D)$, the $\ell_i$ can be chosen to be function evaluations if we fix any $\varepsilon>0$ and allow $\|\varphi_i\|\le1+\varepsilon$. 

In particular, for $\varepsilon=1/n$, we get
a sampling projection $Pf=\sum_{i=1}^{n} f(x_i)\, \varphi_i$ onto $V_n$ with $\|P\|\le n+1$, where $\varphi_i(x_j)=\delta_{i,j}$. 
This also shows that $P$ is an \emph{interpolating projection}, 
meaning that $Pf(x_i)=f(x_i)$ at the sampling points for all $f\in B(D)$. 
Due to the oversampling, the same does not hold true for the projection in Theorem~\ref{thm:KS-new}.
\end{remark}

\smallskip

\begin{remark}[Discretization of the uniform norm] \label{rem:uniform-disc} 
Asking for (sequences of) point sets $P_n\subset D$ 
such that $\|f\|_\infty \le C_n \max_{x\in P_n} |f(x)|$ 
for all $f\in V_n$ 
is often called \textit{discretization of the uniform norm} on $V_n$, see~\cite{KKT21}.  
{Under additional assumptions on $C_n$ and $\# P_n$, }
the corresponding point sets are sometimes called 
\textit{(weakly) admissible meshes}~\cite{CL08,XN23} 
or, if $C_n\asymp 1$, 
\textit{norming sets}~\cite{Bos18}. \\
Norming sets with $\#P_n\asymp n$ exist for special {$V_n$,} 
like polynomials, see~\cite{Bos18,XN23}. 
However, for general $V_n$, 
the relation $\log(\#P_n)\asymp n$ 
is necessary (\cite[Thm.~1.2]{KKT21}) and sufficient (\cite[Thm.~1.3]{KKT21}). 
If one insists in $\#P_n\asymp n$, then it is known that one can choose {$C_n=n+\varepsilon$ (with any $\varepsilon>0$)} for 
general $V_n$,
see~\cite[Prop.~1.2.3]{Novak}. 
This can be improved to $C_n\asymp\sqrt{n}$ for the special case of trigonometric polynomials~\cite[Thm.~1.1]{KKT21} or, more generally, for $V_n$ that satisfy 
a Nikol'skii inequality, see Remark~\ref{rem:Nikolskii}. 
Theorem~\ref{thm:disc} 
shows that the latter condition is not needed. 
\end{remark}

\smallskip

\begin{remark}[Nikol'skii inequalities] \label{rem:Nikolskii}
{There are several papers that contain results 
that are more or less similar to our discretization and approximation results,
but where the factor $\sqrt{n}$ is replaced by 
\begin{equation*}
H(V_n,\mu) \,:=\, \sup_{f\in V_n \setminus \{0\}} \frac{\|f\|_\infty}{\|f\|_{L_2(\mu)}} 
\end{equation*}
for some probability measure $\mu$ on $D$.}
See especially  
{\cite[Thm. 2.12]{DPTT19},} 
\cite[Thm.~2.3]{Kosov_Temlyakov23}, 
\cite[Thm.~1.5 \& Sect.~6.3]{KKT21},
and~\cite[Cor.~5.3]{PU}. 
{The resulting inequality
\[
\|f\|_\infty \,\le\, H(V_n,\mu)\cdot \|f\|_{L_2(\mu)},
\quad f\in V_n,
\]
is called a Nikol'skii-type inequality.
The constant $H(V_n,\mu)$
can be expressed in terms of any $L_2(\mu)$-orthonormal basis $b_1,\hdots,b_n$ of $V_n$ and satisfies
\[
 H(V_n,\mu) \,=\, \sup_{x\in D} \sqrt{|b_1(x)|^2+ \hdots + |b_n(x)|^2}
 \,\ge \sqrt{n}.
\]
In the above results, 
the probability measure $\mu$ can be chosen freely. For these results, the theorem of} 
Kiefer and Wolfowitz~\cite{KieWo60} 
presents itself the missing piece,
as it shows that, {for every $V_n\subset B(D)$,} there 
exists a probability measure $\mu$  {with $H(V_n,\mu) \le \sqrt{n} +\varepsilon$}. 
\end{remark}


\smallskip

\begin{remark}[Constant] \label{rem:const}
The constant $C$ in 
Theorems~\ref{thm:KS-new}
and \ref{thm:LS_sup}
can be chosen as $C=60$. In fact, one can replace the oversampling constant 2 by any smaller constant $b> 1+1/n$ and 
then take $C=60 \cdot (b-1)^{-3/2}$, see \cite[Thm.~6.2]{BSU}. 
\end{remark}


\medskip


\section{On maximizers of the Gram determinant} 
\label{sec:KW}

The following section is based on results by Kiefer and Wolfowitz from \cite{KieWo60}. For the convenience of the reader we give a self-contained proof of the results relevant for our purpose. In addition, we comment on some extensions of the results for complex-valued functions and on the discreteness 
of the measure.  

\medskip

In this section, we consider the set of atomic probability measures
\begin{equation}\label{eq:def_measures}
 \mathcal M \,:=\, \left\{ \sum_{i=1}^N \lambda_i \delta_{x_i} 
 \colon
 N\in \mathbb{N},\, \lambda_k\, \in [0,1],\, \sum_{k=1}^N \lambda_k = 1,\, x_k\in D \right\}.
\end{equation}
Given an $n$-dimensional subspace $V_n$ of complex-valued functions from $B(D)$, 
the goal is to prove a discretization result of the form 
\begin{equation}\label{eq:discr}
    \|f\|_{\infty} \leq \sqrt{n+\varepsilon}\, 
    \Vert f \Vert_{L_2(\varrho)}
    \ ,\quad f\in V_n\,,
\end{equation}
for some $\varrho \in \mathcal M$.

\medskip

The idea of the construction can roughly be described as follows: \\
Given a basis $F = (f_1,\hdots,f_n)^\top$ of $ V_n$, 
we look for a probability measure $\varrho \in \mathcal M$
which maximizes the determinant of the Gram matrix 
\begin{equation}\label{eq:definition-Gram}
G_\varrho[F] \,:=\, \left( \langle f_i, f_j \rangle_{L_2(\varrho)} \right)_{i,j \le n}.
\end{equation}
It was observed already in \cite{KieWo60}
that $\varrho$ maximizes 
the determinant of $G_\varrho[F]$ 
iff it minimizes
the (inverse of the) Christoffel function
\[
    \sup\limits_{f \in V_n \setminus\{0\}} \|f\|_{\infty} /\|f\|_{L_2(\varrho)}\,,
\]
which is the problem that we aim to resolve.
The corresponding measure $\varrho$ will satisfy~\eqref{eq:discr}. 

{As noted (without proof) in \cite{KieWo60},}
we can choose the measure $\varrho \in \mathcal M$ that satisfies \eqref{eq:discr} supported on at most $N\le n^2+1$ sampling points.
Based on a subsampling approach originating in~\cite{BSS}, 
we will then show 
in Section~\ref{sec:disc} that \eqref{eq:discr}
holds, up to a constant factor, for  
a measure of the form $\varrho = \frac{1}{2n} \sum_{i=1}^{2n} \delta_{x_i}$ with {only $2n$ sampling points} $x_i \in D$. 

\begin{remark}
The above procedure is reminiscent of Fekete points~\cite{Bos18},
which are obtained by maximizing the determinant
of the Gram matrix~\eqref{eq:definition-Gram} over all
measures of the form $\varrho = \frac1n \sum_{i=1}^n \delta_{x_i}$ with $x_1,\hdots,x_n\in D$. 
{Here, we are looking for a  maximizer within the set $\mathcal M$ of atomic probability measures with finite support.
In general,}
a probability measure that maximizes the determinant
of the Gram matrix~\eqref{eq:definition-Gram} is called 
an \emph{(approximate) D-optimal design},
{see, e.g., \cite{BKPSU2024,Bos22,HL25,Karlin_Studden1966}.}
\end{remark}

\smallskip

\begin{prop}\label{prop:BD} Let $D$ be a set 
and $V_n$ be an $n$-dimen\-sion\-al subspace of~$B(D)$. 
For every $\varepsilon>0$ and $n\in\N$, 
there exists a natural number 
$N\le
\,\dim({\rm span}(V_n\cdot \overline{V_n}))+1$,
distinct points $x_1,\dots,x_{N}\in D$
and non-negative weights $(\lambda_k)_{k=1}^N$ satisfying $\sum_{k=1}^N \lambda_k = 1$ 
such that, for all $f\in V_n$, we have
$$
 \|f\|_{\infty} \,\leq\, \sqrt{n+\varepsilon} \,\bigg(\sum_{k=1}^{N} \lambda_k \abs{f(x_k)}^2 \bigg)^{1/2}.
$$ 
If $V_n\subset C(D)$ for some compact topological space $D$,
then the conclusion also holds for $\varepsilon=0$. 
\end{prop}

%

\smallskip

{Note that $\dim({\rm span}(V_n\cdot \overline{V_n})) \le n^2$,
but the dimension can be much smaller for particular spaces $V_n$.
A simple example is the $\R$-vector space of univariate polynomials of degree less than $n$, in which case $\dim({\rm span}(V_n\cdot \overline{V_n})) 
=2n-1
$.}
The proof of Proposition~\ref{prop:BD} will make use of the following 
simple (known) lemma.

\smallskip
\begin{lemma}\label{lem1} Let $f_1,...,f_n \colon D \to \C$ be linearly independent functions on a set $D$. Then there are points $x_1,...,x_n \in D$ such that 
the Gram matrix of the uniform distribution on $\{x_1,\hdots,x_n\}$ is positive definite.
\end{lemma}

\begin{proof} Let $F=(f_1,\hdots,f_n)^\top$. The space
$$
V:= {\operatorname{span}} \{F(x) \colon \ x\in D \}\subset \C^n
$$
has a dimension 
$r\leq n$.
Then we have $r$ independent column vectors 
$F(x_1),\dots,F(x_r) \in V$  
such that for any $x\in D$ there exists a unique solution
$\lambda_1(x), \dots, \lambda_r(x)$, for which
\begin{equation*}
\begin{pmatrix}
f_1(x_1)  & f_1(x_2)  & \cdots & f_1(x_r) \\
\vdots  & \vdots  & \ddots & \vdots  \\
f_n(x_1)  & f_n(x_2) & \cdots & f_n(x_r)
\end{pmatrix}\cdot 
\begin{pmatrix}
\lambda_1(x) \\
\vdots \\
\lambda_r(x)
\end{pmatrix}
=
\begin{pmatrix}
f_1(x) \\
\vdots \\
f_n(x)
\end{pmatrix}\,.
\end{equation*}
This gives 
$$
{\operatorname{span}}_{\C} \{f_1, \dots, f_n \}
\subset {\operatorname{span}}_{\C} \{\lambda_1, \dots, \lambda_r\},
$$
so $r\ge n$ and hence $n=r$.
Since the coefficient matrix $M\in \C^{n\times r}$ in the previous linear system has full rank $r$,
it is invertible and
the matrix $\frac1n M M^{*}$
is positive definite. 
It only remains to note that $\frac1n M M^{*}$ equals
the Gram matrix of the uniform distribution on $\{x_1,\hdots,x_n\}$.\\
\end{proof}

%


\smallskip

\begin{proof}[Proof of Proposition~\ref{prop:BD}] {\em Step 1.} Let 
$F = (f_1,\hdots,f_n)^\top$ be a basis of $V_n$, and denote by $F^*$ its conjugate transpose.
For $x\in D$, we define the matrix
 \[
 H(x) := F(x) \cdot F(x)^{*} 
 \in \C^{n\times n},
 \]
 which is the Gram matrix of the Dirac measure $\delta_x$.
It is Hermitian and positive semi-definite 
 and its determinant therefore is real and satisfies $\det H(x) \geq 0$ for all $x\in D$.
We define the set of matrices 
\begin{equation}\label{eq:setH}
  \mathcal{H}:=\{H(x)\colon x\in D\}\,,
\end{equation}
which is bounded in $\C^{n\times n}$ due to the boundedness 
of the functions $f_1,...,f_n$.
Also the convex hull
\begin{align*}
   \convH :&= \bigg\{\sum_{k=1}^N \lambda_k H(x_k)\colon\, N\in \mathbb{N},\, \lambda_k \ge 0,\,\sum_{k=1}^N \lambda_k = 1,\, x_k\in D\bigg\} 
   \\
   &=\, \left\{ G_\varrho[F] \colon \varrho \in \mathcal M \right\}
\end{align*}
 remains a bounded subset of $\C^{n\times n}$. 
 By 
 {Caratheodory's theorem},
 {we can identify 
 $\convH$ with}
 \[
  \convH \,=\, \left\{ G_\varrho[F] \colon \varrho \in \mathcal M \text{ with at most } 
  {r}+1
  \text{ nodes}\right\},
 \]
 {where $r:=\dim({\rm span}(V_n\cdot \overline{V_n}))$,}
 {see \cite[Prop.\,2.8\,(ii) and Lem.\,2.6]{BKPSU2024}}
 {for details}.
The determinant is a continuous mapping on $\C^{n\times n}$. 
Hence, there exists a matrix $A\in \overline{\convH}$, 
i.e., the closure of $\convH$, 
such that 
\[
 \det(A)\,=\, \sup_{M\in\convH}\det(M) \,=\, \sup_{\varrho \in\mathcal M} \det(G_\varrho[F]) \,>\, 0.
\]
The positivity of the determinant follows from Lemma~\ref{lem1}. 

\medskip

{\em Step 2.}
We fix this maximizer of the determinant $A$ and define, 
for 
$x\in D$, $H(x)$ as above and for
$\alpha \in\R$,
the matrix
\begin{equation}\label{eq101}
C_x(\alpha):= (1-\alpha) \cdot A+\alpha \cdot H(x) = A + \alpha(H(x)-A), 
\end{equation}
which is contained in $\overline{\convH}$ for $\alpha \in [0,1]$.
 Clearly 
$$
	p_x(\alpha) = \det [C_x(\alpha)]
$$
is a real-valued polynomial in $\alpha$ and strictly positive for $\alpha \in [0,1)$,  
since the convex combination of a positive definite and positive semi-definite matrix is positive definite.

This polynomial, restricted to $[0,1]$, has 
a maximal value $\det A$ at the point $\alpha =0$, and hence for its derivative it holds
$p'_x(0) \leq 0$. Since this represents the polynomial's coefficient in front of $\alpha$, we get by the fact that the linear coefficient of $\det(I_n + \alpha B)$ is exactly ${\rm tr} B$ and $\det(A+\alpha(H(x)-A)) = \det A\cdot \det(I_n+\alpha A^{-1}(H(x)-A))$ the identity
\begin{equation*}
p_x'(0) = 
\det A \cdot {\rm tr}\big(A^{-1}(H(x)-A)\big)\,.
\end{equation*} 
By the invariance of the trace under circular shifts we obtain
\begin{align*}
{\rm tr}\big(A^{-1}(H&(x)-A)\big)  \,=\, {\rm tr}(A^{-1}H(x)-I_n)
\\
&=\, {\rm tr}(A^{-1}\cdot F(x)\cdot F(x)^*)-n
 \,=\, F(x)^* A^{-1} F(x) - n.
\end{align*}
From $p_x'(0) \le 0$ and $\det A>0$ it follows for all $x\in D$ that 
\be\label{FAF_leq_n}
 F(x)^* A^{-1} F(x) \leq n.
\ee

\medskip

{\em Step 3.}
Since $A\in \overline{\convH}$, there is a sequence of measures $\varrho_j \in \mathcal{M}$, each with at most {$
{r}+1$}
nodes, such that $G_{\varrho_j}[F]$ converges to $A$ for $j\to \infty$.
By continuity of the determinant and of matrix inversion,
this implies that $\det(G_{\varrho_j}[F]) \to \det(A)$ and $G_{\varrho_j}[F]^{-1} \to A^{-1}$.
In particluar, for any $\varepsilon>0$,
since $F(x)$ is bounded, we may choose $j$ large enough such that $\varrho=\varrho_j$ fulfills $\det(G_{\varrho_j}[F])>0$
and $\Vert G_{\varrho}[F]^{-1} - A^{-1} \Vert < \varepsilon/K$, where $K:=\sup_{x\in D} \Vert F(x) \Vert_2^2$.
This implies
\[
\sup_{x\in D}\, \left|F(x)^* \left(G_\varrho[F]^{-1}-A^{-1}\right) F(x)\right| \,<\, \varepsilon.
\]
Hence, using \eqref{FAF_leq_n}, we obtain for all $x\in D$ that
\begin{equation*}
     F(x)^* G_\varrho[F]^{-1} F(x)
     \,\le\, n + \varepsilon\,.
\end{equation*}

\smallskip


{\em Step 4.}
Under a change of basis $T\in\C^{n\times n}$ the Gram matrix becomes
\[
 G_\varrho[TF] \,=\, T \cdot G_\varrho[F] \cdot T^*.
\]
Thus, if we put $T:=G_\varrho[F]^{-1/2}$ (which is the unique positive definite square-root of $G_\varrho[F]^{-1}$ and hence $T^*=T$)
with $\varrho$ from Step~3, we get
\[
 G_\varrho\big[TF\big] \,=\, G_\varrho[F]^{-1/2} \cdot G_\varrho[F] \cdot G_\varrho[F]^{-1/2}
 \,=\, I_n,
\]
i.e., $TF$ is an orthonormal basis of $V_n$ in $L_2(\varrho)$.
Let $B=(b_1,\hdots,b_n)^\top$
be any orthonormal basis of $V_n$. 
For any $x\in D$ and $f=\sum_{k=1}^n c_k b_k$ with $c_k \in \C$,
it follows from the Cauchy-Schwarz inequality that
$\vert f(x)\vert \le \Vert f \Vert_{L_2(\varrho)} \cdot \Vert B(x) \Vert_2$ with equality for $c_k = b_k(x)$.
Hence,
\[
\sup\limits_{f \in V_n\setminus\{0\}} \frac{\|f\|_\infty^2}{\|f\|_{L_2(\varrho)}^2} 
\,=\, \sup_{x\in D} \sup\limits_{f \in V_n\setminus\{0\}} \frac{|f(x)|^2}{\|f\|_{L_2(\varrho)}^2} 
\,=\, \sup_{x\in D} \Vert B(x) \Vert_2^2. 
\]
By plugging in the basis $B=TF$, this equals
\[
\sup_{x\in D}\, (T F(x))^* (T F(x))
 \,=\, \sup_{x\in D}\, F(x)^* G_\varrho[F]^{-1} F(x)
 \,\le\, n + \varepsilon,
\]
which concludes the proof.

Note that in the case $A\in\convH$ (e.g., for continuous $F$ on a compact topological space $D$) the proof above simplifies. We immediately
put $T:=A^{-1/2}$  and consider the respective basis $B=TF$. In view of (\ref{FAF_leq_n}), we get that  $\varepsilon=0$.
\end{proof}

\medskip

\section{Discretization via subsampling}\label{sec:disc}
We combine Proposition~\ref{prop:BD} with a result from \cite{BSU} 
on the discretization of the $L_2$-norm. 
For this we need to restrict to measurable functions, 
and therefore denote, for a given measure space $(D,\mu)$, 
by $B_\mu(D) \subset B(D)$ the space of 
bounded $\mu$-measurable functions with the $\sup$-norm.

\begin{lemma}[{\cite[Thm.~6.2]{BSU}}] \label{lem:disc2}
    Let $(D,\mu)$ be a probability space and let $V_n$ be an $n$-dimensional subspace of $B_\mu(D)$. 
    Then there are points $x_1,\hdots,x_{2n} \in D$
    such that, for all $f\in V_n$, we have 
    \[
     \Vert f \Vert_{L_2(\mu)} 
     \,\le\, 
     57 \, \bigg(\frac{1}{2n} \sum_{k=1}^{2n} |f(x_k)|^2\bigg)^{1/2}.
    \]
\end{lemma}

\begin{proof}
    We apply \cite[Thm.~6.2]{BSU} with $b=2$ and $t$ close to zero. 
    (Note that \cite[Thm.~6.2]{BSU} 
    incorrectly omitted the assumption 
    that $\mu$ is a  probability measure.)\\
\end{proof}

We refer to {
\cite{DPTT19,Gr19,KKLT,Limonova_Malykhin_Temlyakov24,Rud99,T23}} for some history and applications of such discretization results, 
and~\cite{LT} for a direct predecessor. 
Many
of these results deal with \emph{two-sided} bounds between continuous and discrete norms. 
The point set in Lemma~\ref{lem:disc2} is obtained by first sampling $\mathcal{O}(n\log n)$ i.i.d.\ random points and a constructive sub-sampling procedure originating in~\cite{BSS} and further elaborated in \cite{BSU}.

\subsection{Discretization results for the sup-norm} 
We start with the proof of Theorem~\ref{thm:disc}, which we restate here for convenience.

\begin{theorem}\label{thm:disc_sup} 
Let $D$ be a set and $V_n$ be an $n$-dimensional subspace of $B(D)$.
Then there are points $x_1,\hdots,x_{2n} \in D$ such that, for all $f\in V_n$, we have
$$
\|f\|_{\infty} \,\leq\, 58 \sqrt{n} \,\bigg(\frac{1}{2n} \sum_{k=1}^{2n} \abs{f(x_k)}^2 \bigg)^{1/2}\,
\leq 42 \bigg(\sum_{k=1}^{2n} \abs{f(x_k)}^2 \bigg)^{1/2}
.
$$
%
%
\end{theorem}

\medskip

\begin{proof}[Proof of Theorems~\ref{thm:disc} and \ref{thm:disc_sup}]
Proposition~\ref{prop:BD} gives an atomic probability measure $\mu$ such that $\Vert \cdot \Vert_\infty \le t \sqrt{n} \Vert \cdot \Vert_{L_2(\mu)}$ on $V_n$,
where we can choose $t=58/57$.
Lemma~\ref{lem:disc2} further gives an atomic probability measure $\varrho$ of the form $\frac{1}{2n} \sum_{k=1}^{2n} \delta_{x_k}$
such that $\Vert \cdot \Vert_{L_2(\mu)} \le 57 \Vert \cdot \Vert_{L_2(\varrho)}$ on $V_n$. \\
%
\end{proof}

\medskip

It is worth noting that no measure is needed in the formulation of Theorem~\ref{thm:disc_sup}. It only appears in the proof. To illustrate this, let us consider the space of bounded sequences $\ell_\infty := \ell_\infty(\mathbb{N})$. 
Theorem \ref{thm:disc_sup} then asserts that 
for any finite dimensional subspace $V_n$ of $\ell_\infty$ we have $2n$ indices $(i_k)_{k=1}^{2n} \subset \mathbb{N}$ with 
$$
    \sup_{i \in \N}|x_i| \;\leq\; 
   42
    \left(
    \sum\limits_{k=1}^{2n} |x_{i_k}|^2
    \right)^{\frac{1}{2}}, 
    \quad \text{ for all }\quad (x_i)_{i=1}^\infty \in V_n \subset \ell_\infty\,.
$$

\medskip

\subsection{Discretization results for the $L_p$-norm}
By interpolation, we get a discretization result similar to Theorem \ref{thm:disc_sup} also for the $L_p$-norm, $2\leq p \leq \infty$. 
For this, we use the set $Z=P\cup X$ consisting of at most $4n$ points, 
where $P$ is the set of $2n$ ``good'' points for the $L_2$-discretization from Lemma~\ref{lem:disc2} 
and $X$ is the set of
$2n$ ``good'' points for the 
sup-norm from
Theorem \ref{thm:disc_sup}. 

\begin{theorem}\label{thm:disc_Lp}
Let $(D,\mu)$ be a probability space and let $V_n$ be an $n$-dimensional subspace of $B_\mu(D)$, the space of bounded measurable 
functions on $D$.
Then there are points $x_1,\hdots,x_{4n} \in D$
such that, for all $f\in V_n$
and all $2\le p \le \infty$, we have
$$
\|f\|_{L_p(\mu)} \leq
83 \cdot  n^{\frac{1}{2}-\frac{1}{p}} \cdot \Big(\frac{1}{4n} \sum_{k=1}^{4n} \abs{f(x_k)}^2 \Big)^{1/2}\,.
 $$
\end{theorem}

\medskip

Again, the oversampling factor $4$ may be reduced leading to a larger constant, see Remark~\ref{rem:const}. 
However, since the proof relies on the union of two point sets, our method does not allow for 
a factor smaller than~2.

\medskip

\begin{proof}
  Interpolation for $2\leq p \leq \infty$ gives for all $f\in V_n$  that
$$ \Vert  f \Vert _{L_p(\mu)}  \leq 
 \Vert  f  \Vert_{L_2(\mu)}^{\frac{2}{p}} \cdot  \Vert  f \Vert_{\infty}^{1-\frac{2}{p}} .
$$
If we let $x_1,\hdots,x_{2n}$ be the points from Lemma~\ref{lem:disc2} with corresponding measure $\varrho_1=\frac{1}{2n} \sum_{k=1}^{2n} \delta_{x_k}$ and $x_{2n+1},\hdots,x_{4n}$ be the points from Theorem~\ref{thm:disc_sup} with corresponding measure $\varrho_2=\frac{1}{2n} \sum_{k=2n+1}^{4n}\delta_{x_k}$, we get 
$$ \Vert  f \Vert _{L_p(\mu)}  \,\leq\, 58n^{\frac12 - \frac1p} \cdot
 \Vert  f  \Vert_{L_2(\varrho_1)}^{\frac{2}{p}} \cdot  \Vert  f \Vert_{L_2(\varrho_2)}^{1-\frac{2}{p}} .
$$
It remains to observe that for $\varrho = \frac{\varrho_1 + \varrho_2}{2}$,
we have $\Vert \cdot \Vert_{L_2(\varrho_i)}\le \sqrt{2} \Vert \cdot \Vert_{L_2(\varrho)}$,
so that 
\[
 \Vert  f \Vert _{L_p(\mu)}  
 \,\leq\, 58\sqrt{2}\, n^{\frac12 - \frac1p} \cdot
 \Vert  f  \Vert_{L_2(\varrho)}.
\]
\end{proof}

\section{$L_p$-errors of least-squares algorithms}
\label{sec:LS}

We now consider the error of least-squares approximations based on the sampling points from above. 
We refer to~\cite{CL08,CM,Gr19,KKLT} and references therein for results on least-squares in related settings.

As above we consider an $n$-dimensional subspace $V_n$ 
of the space $B(D)$ of bounded complex-valued functions on a set $D$.
Further, let $\varrho$ be a discrete probability measure of the form
\[
 \varrho \,=\, \sum_{x\in X} \lambda_x \delta_x
\]
with a finite multiset $X\subset D$
and $\lambda_x \ge 0$ such that $\sum_{x\in X} \lambda_x =1$,
i.e., $\varrho \in \mathcal M$ with $\mathcal M$ as defined in \eqref{eq:def_measures}.
The expression $\langle \cdot, \cdot \rangle_{L_2(\varrho)}$ is a positive semi-definite Hermitian form on $B(D)$.
We assume that there is some $K>0$ such that
\begin{equation}\label{cond:disc}
\Vert f \Vert_\infty \le K \Vert f \Vert_{L_2(\varrho)}\ ,
\quad f\in V_n.
\end{equation}
Then, in particular, $\Vert \cdot \Vert_{L_2(\varrho)}$ is a norm on $V_n$.
Hence, there exists a unique orthogonal projection 
$A(\varrho,V_n)$ on $B(D)$ with range $V_n$
and it takes the form
of a weighted least squares algorithm,
\[
  A(\varrho,V_n)(f) 
  \,=\,  \underset{g \in V_n}{\rm argmin}\ \Vert g-f \Vert_{L_2(\varrho)}
  \,=\, \underset{g \in V_n}{\rm argmin}\ \sum_{x \in X} \lambda_x\, |g(x)-f(x)|^2.
\]
As such, it is linear and satisfies
\[
 \Vert A(\varrho,V_n) f\Vert_{L_2(\varrho)} \,\le\, \Vert f \Vert_{L_2(\varrho)},
 \quad f\in B(D).
\]
Moreover, since $A(\varrho,V_n) f$ only depends on the values of $f$ at $X$,
there exist functions $\varphi_x \in V_n$
such that
\[
 A(\varrho,V_n)(f)  = \sum_{x\in X} f(x) \varphi_x,
 \quad f\in B(D).
\]
In the case of equal weights,
i.e., for the measure $\varrho_X := \frac{1}{|X|} \sum_{x\in X} \delta_x$,
we 
obtain the norm $\Vert \cdot \Vert_X := \Vert \cdot \Vert_{L_2(\varrho_X)}$ from \eqref{eq:def-seminorm}
and write
\[
A(X,V_n)(f) \,:=\, A\left( \varrho_X, V_n \right)(f)
\,=\, \underset{g \in V_n}{\rm argmin}\ \sum_{x \in X} |g(x)-f(x)|^2.
\]

If we choose $X$ as the discretization set in Theorem~\ref{thm:disc_sup},
the operator $A(X,V_n)$ provides the sampling projection of Theorem~\ref{thm:KS-new}.
More generally, we get the following.
\smallskip

\begin{lemma} \label{lem:sampling-projections-general}
    Let $D$ be a set and $V_n$ be a finite-dimensional subspace of $B(D)$.
    Moreover, let $\varrho \in \mathcal M$ satisfy \eqref{cond:disc}.
    Then $A(\varrho,V_n)$ is a projection onto $V_n$ with $\Vert A(\varrho,V_n) \Vert_{\mathcal L(B(D))} \le K$.
\end{lemma}

\begin{proof}[Proof of Lemma~\ref{lem:sampling-projections-general} and Theorem~\ref{thm:KS-new}]
    As $\Vert \cdot \Vert_\infty$ is a norm on $V_n$, also $\Vert \cdot \Vert_{L_2(\varrho)}$ must be a norm
    and thus $P=A(\varrho,V_n)$ is a well-defined linear projection onto $V_n$.
    Moreover, we have for all $f\in B(D)$ that
    \[
     \Vert Pf \Vert_\infty \,\le\, K \Vert Pf \Vert
     _{L_2(\varrho)}
     \,\le\, K \Vert f \Vert
     _{L_2(\varrho)}
     \,\le\, K \Vert f \Vert_\infty.
    \]
    
    To arrive at Theorem~\ref{thm:KS-new}, we choose $\varrho$ as the uniform distribution on the $2n$-point set from Theorem~\ref{thm:disc_sup}.
\end{proof}

\begin{remark}[Sampling projections with {norm close to $\sqrt{n}$}] 
    We can use the measure $\varrho$ from Proposition~\ref{prop:BD} instead of the points from Theorem~\ref{thm:disc_sup}
    and get that the weighted least squares algorithm $A(\varrho,V_n)$, which uses at most $n^2+1$ nodes, is a sampling projection with norm at most $(1+\varepsilon)\sqrt{n}$
    for any given $\varepsilon>0$.
\end{remark}

\smallskip

We now discuss the error of the least-squares algorithm $A(X,V_n)$ in different norms $\Vert \cdot \Vert$.
We require that the norm  is dominated by the sup-norm 
and that the norm $\Vert \cdot \Vert_X$ from \eqref{eq:def-seminorm} provides good discretization for this norm on $V_n$.
Then one easily gets the following error bound for the least-squares algorithm in 
the norm $\Vert \cdot \Vert$.
Note that the case $\Vert \cdot \Vert = \Vert \cdot \Vert_{L_2(\mu)}$ is contained already in \cite{T20} (with $(1+C)$ replaced by $(1+2C)$).
{We also note that the constant $C$ may depend on the dimension $n$ of the space $V_n$.}

\begin{lemma}\label{lem:general-norm}
    Let $D$ be a set and $G$ be a subspace of $B(D)$
    and let $\Vert \cdot \Vert$ be a seminorm on $G$
    such that $\Vert \cdot \Vert \le \Vert \cdot \Vert_\infty$ on $G$.
    Moreover, let $V_n$ be a finite-dimensional subspace of $G$
    and let $X \subset D$ be a finite multiset
    such that $\Vert \cdot \Vert_X$ is a norm on $V_n$
    and $\Vert \cdot \Vert \le C \Vert \cdot \Vert_X$ on $V_n$ with some $C>0$.
    Then, the least-squares algorithm $A(X,V_n)$ is well defined on $B(D)$ and, for all $f\in G$, we have
    \[
     \Vert f - A(X,V_n)f \Vert
     \;\le\;
     (1+C) \,\inf_{g\in V_n} \Vert f - g \Vert_\infty.
    \]
    \end{lemma}

   \begin{proof}  We obtain for any $f\in G$ and any $g\in V_n$ that
 \begin{align*}
     \Vert f - A&(X,V_n) f \Vert
     \,\le\, \Vert f - g \Vert + \Vert g - A(X,V_n) f \Vert\\
     &\le\, \Vert f - g \Vert_\infty + 
     C\, \Vert g - A(X,V_n) f \Vert_X\\
     &=\, \Vert f - g \Vert_\infty + 
     C\,  \Vert A(X,V_n)(g -  f) \Vert_X\\
     &\le\, \Vert f - g \Vert_\infty +  C\, \Vert g - f \Vert_X\\
     &\le\, (1+C)\cdot \Vert g - f \Vert_\infty.
 \end{align*}
 The statement is obtained by taking the infimum over $g\in V_n$. \\
\end{proof}


Together with our discretization result for the sup-norm, 
this implies the error bound that we stated in the introduction.

\begin{proof}[Proof of Theorem~\ref{thm:LS_sup}]
Combine Theorem~\ref{thm:disc_sup} and 
Lemma~\ref{lem:general-norm} with $G=B(D)$ 
and $\|\cdot\|:=\|\cdot\|_\infty$.\\
\end{proof}

We can also insert the discretization
result from Theorem~\ref{thm:disc_Lp}
to obtain the following more general estimate
in the $p$-norm.

\begin{theorem}\label{thm:LS_Lp}
Let $(D,\mu)$ be a probability space and let $V_n$ be an $n$-dimensional subspace of $B_\mu(D)$. 
Then there are points $x_1,\hdots,x_{4n} \in D$
such that, for all $f\in B_\mu(D)$ and all $1\le p \le \infty$, we have
\[
     \Vert f - A(X,V_n)f \Vert_{L_p(\mu)}
     \,\le\,
     84\, n^{(1/2-1/p)_+} \,\inf_{g\in V_n} \Vert f - g \Vert_\infty,
\]
where $a_+:=\max\{a,0\}$ for $a\in \R$.
\end{theorem}

\medskip

\begin{proof}
    The case $p\le 2$ follows from the case $p=2$, so let $p\ge 2$.
    In this case, the result is a direct consequence of Theorem~\ref{thm:disc_Lp} and Lemma~\ref{lem:general-norm}
    with $G=B_\mu(D)$
   and $\|\cdot \|=\| \cdot\|_{L_p(\mu)}$.\\
\end{proof}

\begin{remark}[Universality]
    It is worthwhile to note that the algorithm $A(X,V_n)$ from Theorem~\ref{thm:LS_Lp} does not depend on $p$;
    the same least-squares algorithm is good 
    in all $L_p(\mu)$-norms.
\end{remark}

\smallskip

\begin{remark}[Least-maximum estimator]\label{Rem:least_maximum}
Lemma~\ref{lem:general-norm} {with $G=B(D)$} shows that 
 a 
discretization of the sup-norm on $V_n$ 
by the discrete $\ell_2$-norm 
{with a factor smaller than $\sqrt{n}$ (as guaranteed by Theorem~\ref{thm:disc_sup}) also implies a smaller error bound for $A(X,V_n)$.}
A reverse statement is not true.
\footnote{{For $X=D=\{1,\dots,n\}$
and $B(D)=V_n=\R^n$, the method  $A(X,V_n)$ is exact on $B(D)$, but a discretization with factor smaller $\sqrt{n}$ is not possible.}}\\
In contrast, for
the (typically nonlinear) 
\emph{least-maximum estimator}
\[
M(X,V_n)f\,:=\,\underset{g \in V_n}{\rm argmin}\ \max_{x \in X} |g(x)-f(x)|
\]
on $B(D)$,
where we assume that the finite point set $X\subset D$ is such that $g\in V_n$ and $g\vert_X=0$ implies $g=0$,
one can prove that 
each of the statements
\begin{compactenum}
    \item $\Vert g \Vert_\infty \le c_1 \cdot  \max_{x\in X} |g(x)|$ \, for all $g\in V_n$,
    \item $\|M(X,V_n) f \|_\infty \le c_2 \cdot \| f \|_\infty$ \, for all $f\in B(D)$,
    \item $\Vert f - M(X,V_n)f \Vert_\infty \le c_3 \cdot \inf_{g\in V_n} \Vert f - g\Vert_\infty$ \, for all $f\in B(D)$,
\end{compactenum}
implies the other two with a slightly increased  $c_i$, $i\le 3$. 
In particular, a good discretization on $V_n$ by the discrete sup-norm, see Remark~\ref{rem:uniform-disc}, is equivalent to a small error of $M(X,V_n)$.
\end{remark}

For convenience, let us present the proof of the above equivalences.

\begin{proof}[Proof] 
We write $\Vert f \Vert_* = \max_{x\in X} |f(x)|$ and $M=M(X,V_n)$.\\[3pt]
(1) $\Rightarrow$ (3):
For any $f\in B(D)$ and $g\in V_n$,
\[
 \Vert g - Mf \Vert_\infty
 \le c_1 \Vert g - Mf \Vert_*
 \le c_1 (\Vert g - f \Vert_*
 +  \Vert f - Mf \Vert_*)
 \le 2 c_1 \Vert f - g \Vert_\infty
\]
and hence
\[
 \Vert f - Mf \Vert_\infty
 \le \Vert f - g \Vert_\infty
 + \Vert g- Mf \Vert_\infty
 \le (2c_1 + 1) \Vert f - g \Vert_\infty.
\]
(3) $\Rightarrow$ (2):
This follows from
\[
 \Vert M f \Vert_\infty
 \le \Vert f \Vert_\infty + \Vert f-Mf \Vert_\infty
 \le (1+ c_3) \Vert f \Vert_\infty.
\]
(2) $\Rightarrow$ (1):
Let $g\in V_n$. We choose $f \in B(D)$ with $f(x)=g(x)$ for $x\in X$ and $f(x)=0$ otherwise.
Then $\Vert f - Mf \Vert_* \le \Vert f - g \Vert_*=0$ and hence $Mf\vert_X=f\vert_X=g\vert_X$
so that $Mf=g$.
This gives
\[
 \Vert g \Vert_\infty 
 = \Vert Mf \Vert_\infty
 \le c_2 \Vert f \Vert_\infty
 = c_2 \Vert g \Vert_{*}. 
\]
\end{proof}



\section{Implications for optimal sampling recovery}
\label{sec:OR}

The above results have
consequences for the problem of optimal recovery 
based on function evaluations, 
which was actually the starting point for our studies, see~\cite{KPUU23}.

For this, let us define the \emph{$n$-th linear sampling number} of a class 
$F$ of complex-valued functions on a set $D$ 
in a (semi-)normed space $G$ by
\begin{align*}
g_n^{\lin}(F,G) \,:=\, 
\inf_{\substack{x_1,\dots,x_n\in {D}\\ \varphi_1,\dots,\varphi_n\in G}}\, 
\sup_{f\in F}\, 
\Big\|f - \sum_{i=1}^n f(x_i)\, \varphi_i\Big\|_G.
\end{align*}
This is the \emph{minimal worst case error} that 
can be achieved with 
a linear sampling recovery algorithm
based on at most $n$ function values, 
if the error is measured in $G$. 

We want to compare the \emph{power} of linear sampling algorithms and arbitrary algorithms based on arbitrary information. 
For this, 
%
one usually takes the Gelfand numbers of the class $F$ as the ultimate benchmark, 
see Section~\ref{sec:gn-cn}. 
However, we start with a rather direct reformulation of our results in terms of the Kolmogorov numbers. 
These comparisons will
also imply the sharpness of the factor $\sqrt{n}$ in our main results.

\subsection{Sampling vs.~Kolmogorov numbers} \label{sec:gn-dn}

The \emph{$n$-th Kolmogorov number} 
of a subset $F$ of a (semi-)normed space $G$
is defined by 
\[
d_n(F,G) \;:=\; \inf_{\substack{V_n\subset G\\ \dim(V_n)=n}}\sup_{f\in F}\,\inf_{g \in V_n}\|f-g\|_G.  
\]
This quantity often appears as a benchmark in the literature.  
However, although this is a natural geometric quantity, let us stress that $d_n$ has in general no relation to optimal algorithms, see e.g.~\cite{NW1}.

A general bound in this context is due to Novak, see~\cite[Prop.~1.2.5]{Novak} or \cite[Thm.~29.7]{NW3}, 
which states that 
\begin{equation}\label{eq:Novak}
g_n^\lin(F,B(D)) 
\,\le\, (1+n)\,d_n(F,B(D)).
\end{equation}
If we stick to $n$ samples,
the factor $(1+n)$
cannot be improved:
For each $n$, there is a class $F$ where \eqref{eq:Novak} is an equality, 
see~\cite[Ex.~29.8]{NW3}.
On the other hand, 
a result by Kashin, Konyagin, and Temlyakov, see \cite[Thm.~5.2]{KKT21}, shows that
for {compact}
sets $F\subset C(D)$ on 
compact
$D \subset \R^d$, 
the factor $n$ can be removed by exponential oversampling, namely,
\begin{equation*}\label{eq:KKT}
g_{9^n}^\lin(F,B(D)) 
\,\le\, 5\,d_n(F,B(D)).
\end{equation*}
Actually, the relation above was proved for 
the least-maximum estimator (see the definition in Remark \ref{Rem:least_maximum}), which is non-linear. But one knows (see, e.g., 
\cite[Thm.\,4.8]{NW1} or \cite[Thm.\,1 and Prop.\,13]{Creutzig_Wojtaszczyk2004}) that
if 
$F$ is convex and symmetric and
the target space is $B(D)$ (as in our case), 
the non-linear and the linear sampling numbers are equal.
Applying this equality for the convex hull $K$ of $F\cup (-F)$
and using that $d_n(K, B(D))=d_n(F,B(D))$ gives the above inequality for general compact $F$.

We present a result that counterbalances the oversampling and the extra factor in the error bound. Namely,
Theorem~\ref{thm:LS_sup} 
implies that
\eqref{eq:Novak} can be improved by a factor of $\sqrt{n}$ merely by allowing a constant oversampling factor, i.e., $b\cdot n$ sampling points instead of $n$. 
From Theorems~\ref{thm:LS_sup} and~\ref{thm:LS_Lp}, we obtain a corresponding result for $L_p$-approximation.

\begin{theorem}\label{thm:widths}
    There is a constant $C>0$ such that the following holds.
    \begin{enumerate}
        \item For any set $D$, any $F\subset B(D)$, and all $n\in\N$,
    \begin{equation}\label{eq:gndn}
  g_{2n}^\lin(F,B(D)) 
\;\le\; C\,\sqrt{n}\;d_n(F,B(D)).
    \end{equation}
    \item For any probability space $(D,\mu)$, any $F\subset B_\mu(D)$, all $n\in\N$, and all $1\le p \le \infty$,
       \begin{equation}\label{eq:gndn_Lp}
  g_{4n}^\lin(F,L_p(\mu)) 
\;\le\; C\,n^{(1/2-1/p)_+}\;d_n(F,B_\mu(D)).
    \end{equation}
    \end{enumerate}
\end{theorem}

\medskip

Note that the case $p=2$ of the previous theorem, and 
of Theorem~\ref{thm:LS_Lp}, was obtained first in~\cite[Thm.~1.1]{T20} for classes of continuous functions on compact domains by using a weighted least-squares method, 
in contrast to the unweighted least-squares method employed here
and in~\cite{BSU}.

\smallskip

\begin{remark}
If $D$ is a compact topological space and $F\subset C(D)$,
then~\eqref{eq:gndn} holds with $B(D)$ replaced by $C(D)$.
Moreover, if $\mu$ is a Borel measure,
we can replace $B_\mu(D)$ with $C(D)$ in \eqref{eq:gndn_Lp} as well.
\end{remark}

\smallskip

\begin{proof}[Proof of Theorem~\ref{thm:widths}]
    {The first part}
    follows immediately from Theorem~\ref{thm:LS_sup} {(or Theorem~\ref{thm:KS-new})}
    and {the second part from} 
    Theorem~\ref{thm:LS_Lp} by taking the supremum over all $f\in F$
    and then the infimum over all $V_n \subset B(D)$, 
    respectively $B_\mu(D)$. \\
\end{proof}

\medskip

\subsection{Sampling vs.~Gelfand numbers} \label{sec:gn-cn}

In this section, we want to discuss the power of linear algorithms using $n$ function evaluations compared to arbitrary (possibly non-linear) algorithms using $n$ pieces of arbitrary continuous linear information. In other words, we discuss
the relation of 
linear sampling numbers and Gelfand numbers. 
We assume that
$F$ is the unit ball of a normed function space $\widetilde{F}$, where point evaluations are continuous, and which is continuously embedded into $G$, see for instance \cite[Def.\ 2.6 and Lem.\ B.1]{JUV}. 
Note that point evaluations are continuous if 
$\widetilde{F}$ is continuously embedded into $B(D)$.

The 
\emph{$n$-th Gelfand number} 
of $F$ in $G$ is defined by
\begin{align*}
c_n(F, G) \,:=\,
\inf_{\substack{\psi\colon \C^n\to G\\ 
N\in \mathcal{L}(\widetilde{F},\C^n)}
}\; 
\sup_{f\in F}\, 
\big\|f - \psi\circ N(f) \big\|_{G}.
\end{align*}
It measures the worst case error of the optimal (possibly non-linear) algorithm using $n$ pieces of arbitrary linear information.

It is well known that linear algorithms are optimal if $G=B(D)$ and hence
\begin{equation}\label{eq:dncn}
d_n(F,B(D))\,\le\, c_n(F,B(D))
\,=\, \inf_{\substack{T\in \mathcal{L}(\widetilde{F},B(D)) \\ 
{\rm rk}(T)\le n}}\; 
\sup_{f\in F}\, 
\big\|f - Tf \big\|_\infty,
\end{equation}
see e.g.~\cite[Thm.~4.8]{NW1} and the reference there. 
Hence, we obtain in this case from Theorem~\ref{thm:widths} that

\begin{equation}\label{eq:gncn}
  g_{2n}^\lin(F,B(D)) 
\;\le\; C\,\sqrt{n}\;c_n(F,B(D))\,,
\quad n\in\N.
\end{equation}
This and the example provided in Section~\ref{sec:sharpness} show that turning from arbitrary algorithms for uniform approximation 
to linear algorithms based on function values, 
the maximal loss in the algebraic rate of convergence is equal to 1/2.

We do not know of a similar result for $L_p$-approximation in the case $p<\infty$.
Theorem~\ref{thm:widths} does not admit a direct comparison of the numbers $g_{n}^\lin(F,L_p(\mu))$ and $c_n(F,L_p(\mu))$.
It is an interesting open problem what the maximal loss in the rate of convergence can be for $p<\infty$.

An exception is the case $p=2$. Here, recent results show
that the same algebraic rate is obtained 
by $g_n^\lin(F,L_2(\mu))$ and $c_n(F,L_2(\mu))$ 
if $F$ is the unit ball of a reproducing kernel Hilbert space  
and the numbers $c_n(F,L_2(\mu))$ are square-summable, 
see~\cite{DKU,KU1,NSU}. 
This summability condition is sharp, see \cite{HNV,HKNV,KV}. 
For general classes $F$,
the maximal loss in the rate of convergence is again equal to $1/2$:
If $c_n(F,L_2(\mu)) \lesssim n^{-\alpha}$ for some $\alpha>1$, then $g_n^\lin(F,L_2(\mu))\lesssim n^{-\alpha+1/2}$.
We refer to \cite[Thm.~1\,\&\,2]{KSUW} for quick access to the required formulas.
It is obtained from the Hilbert space results via embedding \cite{DKU,KU2}
and a classical result about $s$-numbers from \cite[Thm.~8.4]{Pietsch-s}, 
see also~\cite{Creutzig_Wojtaszczyk2004}.

The algorithm that is used in the aforementioned papers
is different from the algorithm that we use in the present paper: the sampling points are chosen differently and weights might be required.
Under additional assumptions, however,
this algorithm 
is also a very good algorithm for approximation in other norms, including the uniform and $L_p$ norm.
The corresponding analysis can be found in
\cite{GW23,KPUU23,PU} or~\cite{SU23}.
In particular, for many reproducing kernel Hilbert spaces,
even the factor $\sqrt{n}$ in~\eqref{eq:gncn} can be avoided.

\medskip

\subsection{Sharpness of our results}\label{sec:sharpness}

{We now argue that the upper bounds in Theorem~\ref{thm:KS-new}, Theorem~\ref{thm:disc}, Theorem~\ref{thm:LS_sup}, \eqref{eq:gndn} and  \eqref{eq:gncn} are sharp in the following sense:
The polynomial order of the $n$-dependent (or constant) factor on the right hand side cannot be improved, also if the number of points $2n$ is replaced with $N\le cn^\alpha$ points for any absolute constants $c,\alpha>0$.
If the amount of samples is linear (i.e., $\alpha=1$), the upper bounds cannot even be improved by any factor $o(1)$.

{For this, we first recall {how the different results interact.} 
An improved factor in the discretization bound (Theorem~\ref{thm:disc}) would lead to a correspondingly improved norm of the sampling projection (Theorem~\ref{thm:KS-new}) 
{via Lemma~\ref{lem:sampling-projections-general},}
and {a correspondingly improved} error of the least-squares algorithm (Theorem~\ref{thm:LS_sup})
{via Lemma~\ref{lem:general-norm}}.
On the other hand, an improved version of any of the {Theorems~\ref{thm:KS-new} or \ref{thm:LS_sup}} would lead to an analogously improved bound on the linear sampling numbers in \eqref{eq:gndn}.
Finally, via the relation~\eqref{eq:dncn}, an improved factor in \eqref{eq:gndn} results in an improved factor in \eqref{eq:gncn}.
In all these implications, the number of samples does not change.
It hence suffices to show that \eqref{eq:gncn} is sharp in the two ways described above.
}

%
%

In order to prove the sharpness of the bound \eqref{eq:gncn}, we
choose the class $F \subset C(D)$ 
as the unit ball of the Besov space $B^s_{1,q}(D)$ 
with smoothness $s>1$ and $1\le q\le \infty$
on the domain $D=[0,1]$.
For this example, it is known that
\begin{equation}\label{eq:example}
g_{n}^\lin(F,B(D)) 
\gtrsim n^{-s+1}
\quad\text{and}\quad
c_n(F,B(D)) \lesssim n^{-s+1/2},
\end{equation}
see \cite[eq.\,(1.19)]{NoTr06} and \cite[Thm.\,4.12]{V08}.\footnote{{To be precise,
both references consider $L_\infty(D)$ instead of $B(D)$.
However, one can show that replacing $L_\infty(D)$ with $B(D)$ does not change the numbers $g_n^\lin$ and $c_n$ if $F$ is a convex and symmetric subset of $C(D)$, see \cite[Sec.\,4.2]{NW1}.}}
%
Now, if we could replace the factor $\sqrt{n}$ in \eqref{eq:gncn} with a factor $K_n = o(\sqrt{n})$ in exchange for replacing $2n$ with $cn$ for some constant $c\in\N$,
this would clearly contradict \eqref{eq:example}.
But even if we replace the number of samples $2n$ by $c n^\alpha$ 
with constants $c,\alpha\in\N$,
we cannot replace the factor $\sqrt{n}$ by a factor $n^\beta$ with $\beta<1/2$
as this would also contradict \eqref{eq:example}
in the case that $s>1$ is close to one, namely, for $s< (\alpha-\beta-1/2)/(\alpha-1)$.

If we turn to the case of $L_p$-approximation with $2\le p < \infty$ and consider the same example,
it is known that
\[
g_n^\lin(F,L_p) \gtrsim n^{-s+1-1/p},
\]
see again \cite[eq.\,(1.19)]{NoTr06}. 
A similar reasoning as above shows that the factor $n^{(1/2-1/p)_+}$ cannot be replaced by a lower-order term using a linear amount of samples in any of the results 
\eqref{eq:gndn_Lp}, Theorem~\ref{thm:disc_Lp} and Theorem~\ref{thm:LS_Lp}. 
What the example does not tell us is whether
the factor can be improved by using a polynomial amount of samples.
We would have to choose $s$ close to $1-1/p$ for a contradiction of the formulas, but this is not possible due to the restriction $s>1$.
We leave this as an open problem.
}

\smallskip

\begin{remark}
The estimate \eqref{eq:gndn} is stronger than \eqref{eq:gncn}.
Indeed, {for the unit ball $B^s_2$ of the Sobolev space of functions on $[0,1]$ with smoothness $s>1$ it holds}
    \[
    g_{n}^\lin(B^s_2,L_\infty)
    \asymp c_{n}(B^s_2,L_\infty)
    \asymp n^{-s+1/2},
    \quad
    d_n(B^s_2,L_\infty)\asymp n^{-s},
    \]
    see again~\cite[Sect.~1.3.11]{Novak} 
    and~\cite[Thm.~VII.1.1]{Pinkus85} (or \cite{Kashin77}).
    Hence, \eqref{eq:gndn}
gives the optimal bound on the sampling numbers 
for the 
class $B_2^s$, while \eqref{eq:gncn} would lead to a suboptimal bound.
\end{remark}



\noindent\textbf{Acknowledgement.} The authors would like to thank Abdellah Chkifa and Matthieu Dolbeault for bringing the reference \cite{KieWo60} to their attention, see~\cite{CD23}.
In addition, the authors would like to thank Erich Novak for pointing out the connection of Theorem~\ref{thm:LS_sup} to the classical Kadets-Snobar theorem and Auerbach's lemma
{as well as an anonymous referee for several helpful comments that added a lot more clarity to the paper}. 
The authors are also grateful for the hospitality of the Leibniz Center Schloss Dagstuhl where this manuscript has been discussed during the Dagstuhl Seminar 23351 {\it Algorithms and Complexity for Continuous Problems} (August/September 2023). 
DK is supported by the Austrian Science Fund (FWF) grant M~3212-N. 
For open access purposes, the author has applied a CC BY public copyright license to any author accepted manuscript version arising from this submission. 
KP acknowledges support by the Philipp Schwartz Fellowship of the Alexander von
Humboldt Foundation and the budget program of the NAS of Ukraine ``Support for the development of priority areas of research'' (KPKVK 6541230). 
MU is supported by the Austrian Federal Ministry of Education, Science and Research via the Austrian Research Promotion Agency (FFG) through the project FO999921407 (HDcode) funded by the European Union via NextGenerationEU.
TU and KP are supported by the German Research Foundation (DFG) with grant Ul403/4-1\,.



\begin{thebibliography}{77}


\bibitem{BSS} J.D. Batson, D.A. Spielman, N. Srivastava, 
{\it Twice-Ramanujan sparsifiers}, in: STOC’09—Proceedings of the 2009 ACM International Symposium on Theory of Computing, pp. 255–-262, ACM, New York, 2009.

\bibitem{BKPSU2024}
F. Bartel, L. K\"ammerer, K. Pozharska, M. Sch\"afer, T. Ullrich,
{\it
Exact discretization, tight frames and recovery via D-optimal designs},
arXiv:2412.02489, 2024.



\bibitem{BSU} F. Bartel, M. Sch\"afer, and T. Ullrich, {\it Constructive subsampling of finite frames with applications in optimal function recovery}, Applied. Comput. Harmon. Anal. 65, 209--248, 2023.





\bibitem{Bos18}
L.~Bos, 
\textit{Fekete points as norming sets}, 
Dolomites Research Notes on Approximation 11(4), 26--34, 2018. 

\bibitem{Bos22}
L.~Bos, 
\textit{On optimal designs for a $d$-cube},
Dolomites Research Notes on Approximation 15(4), 20--34, 2022.





\bibitem{CL08}
J.-P. Calvi, N. Levenberg, 
Uniform approximation by discrete least squares polynomials, 
J. Approx. Theory 152(1), 82--100, 2008.

\bibitem{CD23}
A. Chkifa and M. Dolbeault, 
{\it Randomized least-squares with minimal oversampling and
interpolation in general spaces}, 
SIAM J. Numer. Anal.  62(4), 1515--1538, 2024.



\bibitem{CM} A. Cohen, G. Migliorati, 
\emph{Optimal weighted least squares methods,} SMAI J. Comput. Math. 3, 181--203, 2017.

\bibitem{Creutzig_Wojtaszczyk2004} 
J. Creutzig, P. Wojtaszczyk,
{\it Linear vs. nonlinear algorithms for linear problems},
J. Complexity 20, 807--820, 2004. 

\bibitem{DPTT19}
F. Dai, A. Prymak, V.N. Temlyakov, S.Yu. Tikhonov, 
{\it Integral norm discretization and related problems}, Russ. Math.
Surv. 74 (4) (2019) 579--630
. Translation from Usp. Mat. Nauk 74:4 (448), 3--58, 2019.


\bibitem{DKU} M.\,Dolbeault, D.\,Krieg, and M.\,Ullrich,  
{\it A sharp upper bound for sampling numbers in $L_2$}, 
{Appl. Comput. Harmon. Anal.} 63, 113--134, 2023. 











\bibitem{GW23}
J.~Geng and H.~Wang, 
{\it On the power of standard information for tractability for $L_\infty$ approximation of periodic functions in the worst case setting}, 
J. Complexity, 80, 101790, 2024.






\bibitem{Gr19} K. Gröchenig, {\it Sampling, Marcinkiewicz-Zygmund inequalities, approximation, and quadrature rules}, J. Approx. Theory, 257:105455, 2020.


\bibitem{HL25}
D. Henrion, J.B. Lasserre. 
Approximate D-optimal design and equilibrium measure.
2025. hal-04688534v2f

 \bibitem{HNV} A.\,Hinrichs, E.\,Novak, and J.\,Vyb\'iral, {\it  Linear information versus function evaluations for $L_2$-approximation}, J. Approx. Theory, 153(1),  97-107, 2008.





\bibitem{HKNV} A.\,Hinrichs, D.\,Krieg, E.\,Novak, and J.\,Vyb\'iral, 
{\it Lower bounds for integration and recovery in $ L_2$}, J. Complexity, 101662, 2022. 



\bibitem{JUV} T. Jahn, T. Ullrich, F. Voigtlaender,
{\it Sampling numbers of smoothness classes via $\ell^1$-minimization},
J. Complexity 79, 101786, 2023.


\bibitem{KS} M.\,I.~Kadets and M.\,G.~Snobar,
{\it Some functionals over a compact Minkovskii space}, Math. Notes, 10(4), 694--696, 1971.

\bibitem{Kashin77} 
B.\,S.\,Kashin, 
{\it Diameters of some finite-dimensional sets and
classes of smooth functions}, 
Izv. Akad. Nauk SSSR Ser. Mat., 41(2): 334--351, 1977; English transl. in Math. USSR-Izv.,
11(2), 317--333, 1977,
https://doi.org/10.1070/IM1977v011n02ABEH001719.

\bibitem{KKT21}
B.~Kashin, S.~Konyagin, and V.~Temlyakov, 
{\it Sampling discretization of the uniform norm},
Constr. Approx. 57(2), 663--694, 2023.

\bibitem{Karlin_Studden1966}
S. Karlin and W.J. Studden,
{\it Tchebycheff systems: with applications in
analysis and statistics},
Wiley Interscience, New York, 1966.

\bibitem{KKLT} B. Kashin, E. Kosov, I. Limonova, and V. Temlyakov,
{\it Sampling discretization and related problems},
J. Complexity, 101653, 2022.



 

\bibitem{KieWo60}
J.~Kiefer and J.~Wolfowitz.
\newblock {\it The equivalence of two extremum problems}, 
\newblock Canadian J. Math. 12, 363--366, 1960.


\bibitem{Kosov_Temlyakov23} 
E.\,D. Kosov and V.\,N. Temlyakov,
{\it Sampling discretization of the uniform norm and applications},
J. Math. Anal. Appl. 538 (2), 128431, 2024.


\bibitem{KPUU23}
D.\,Krieg, K.\,Pozharska, M.\,Ullrich, T.\,Ullrich,
{\it Sampling recovery in $L_2$ and other norms}, 
to appear in Math. Comp., 
arXiv:2305.07539, 2023.


\bibitem{KSUW} D. Krieg, P. Siedlecki, M. Ullrich, and H. Wo\'zniakowski,
{\it Exponential tractability of $L_2$-approximation with function values},
Adv. Comput. Math. 49, 18, 2023. 



\bibitem{KU1} D. Krieg and M. Ullrich,
{\it Function values are enough for $L_2$-approximation}, 
Found. Comp. Math., 21(4), 1141--1151, 2021.

\bibitem{KU2} D.\,Krieg and M.\,Ullrich,
{\it Function values are enough for $L_2$-approximation: Part II}, 
J. Complexity 66, 101569, 2021. 

\bibitem{KV} D.\,Krieg and J.\,Vyb\'iral, {\it New lower bounds for the integration of periodic functions}, 
 J. Fourier Anal. Appl. 29(41), 2023. 












\bibitem{Limonova_Malykhin_Temlyakov24}
I.V. Limonova, Y.V. Malykhin, and V.N. Temlyakov.
One-sided discretization inequalities and sampling recovery.
{\it Russian Math. Surveys}  79(3), 515-545, 2024.

\bibitem{LT} I. Limonova and V. Temlyakov, {\it On sampling discretization in $L_2$}, J. Math. Anal. Appl. 515(2), 126457, 2022. 











\bibitem{NSU} N. Nagel, M. Sch\"afer, and T. Ullrich,
{\it A new upper bound for sampling numbers}, Found. Comp. Math. 22(2), 445-468, 2022.






\bibitem{Novak} E. Novak, {\it Deterministic and stochastic error bounds in numerical analysis},
Lecture Notes in Mathematics 1349, Springer-Verlag, 1988.

\bibitem{NoTr06}
E.~Novak and H.~Triebel, 
\emph{Function Spaces in Lipschitz Domains and
Optimal Rates of Convergence for Sampling}, 
Constr. Approx. (2006) 23: 325–350.




\bibitem{NW1}
E.~Novak and H.~Wo\'zniakowski,
\newblock {\em Tractability of multivariate problems. {V}olume~{I}: {L}inear
  information}, volume~6 of {\em EMS Tracts in Mathematics}.
\newblock European Mathematical Society (EMS), Z\"urich, 2008.


\bibitem{NW3} E.\,Novak and H.\,Wo\'zniakowski, {\it Tractability of multivariate problems. Volume~III: Standard information for operators},
volume~18 of {\em EMS Tracts in Mathematics}.
\newblock European Mathematical Society (EMS), Z\"urich, 2012.






\bibitem{Pie07} A. Pietsch. {\it History of Banach spaces and linear operators}, Birkhäuser Boston, MA, 2007.


\bibitem{Pietsch-s}
A. Pietsch, {\it s-Numbers of operators in Banach spaces}, 
Studia Math. 51, 201--223, 1974.


\bibitem{Pinkus85}
A. Pinkus, {\it n-width in approximation theory}, 
Ergebnisse der Mathematik und ihrer Grenzgebiete. 3. Folge / A Series of Modern Surveys in Mathematics (MATHE3, volume 7), Springer Berlin, Heidelberg, 1985. 


\bibitem{PU}
K. Pozharska and T. Ullrich.
{\it A note on sampling recovery of multivariate functions in the uniform norm}, 
SIAM J. Numer. Anal., 60(3),  1363--1384, 2022. 

\bibitem{Rud99}
M. Rudelson, 
{\it Almost orthogonal submatrices of an orthogonal matrix}, Isr. J. Math. 111, 143--155, 1999.





\bibitem{Schn14} R. Schneider, {\it Convex bodies: the Brunn-Minkowski theory}, 2nd expanded ed. (English),
Encyclopedia of Mathematics and its Applications 151. Cambridge University Press, 2014. 



\bibitem{SU23}
M.~Sonnleitner, M.~Ullrich, 
{\it On the power of iid information for linear approximation}, 
J. Appl. Numer. Anal., 
 1,  88--126, 2023, 
https://doi.org/10.30970/ana.2023.1.88.












\bibitem{T20} V.\,N.\,Temlyakov, {\it On optimal recovery in $L_2$},  
J. Complexity, 65, 101545, 2021. 

\bibitem{T23} V.\,N.\,Temlyakov, {\it 
On universal sampling recovery in the uniform norm}, 
Proc. Steklov Inst. Math. 323, 206–216, 2023. 







\bibitem{V08} J. Vyb\'\i ral, 
{\it Widths of embeddings in function spaces},
J. Complexity, 24(4), 545--570, 2008.

\bibitem{XN23}
Y.~Xu, A.~Narayan, 
\emph{Randomized weakly admissible meshes}, 
J.~Approx. Theory 285, 105835, 2023.








\end{thebibliography}
\end{document}